\documentclass[11pt,a4paper,reqno]{amsart}%
\usepackage[latin1]{inputenc}
\usepackage{mathrsfs}
\usepackage{dsfont}
\usepackage{amsmath}
\usepackage{amssymb}
\usepackage{amsthm}
\usepackage{amsfonts}
\usepackage{amstext}
\usepackage{amsopn}
\usepackage{amsxtra}
\usepackage{mathrsfs}
\usepackage{dsfont}
\usepackage{esint}
\usepackage{pst-all}
\usepackage{graphicx}
\newtheorem{theorem}{Theorem}
\newtheorem{proposition}{Proposition}
\newtheorem{remark}{Remark}

\newtheorem{lemma}{Lemma}
\newtheorem{corollary}{Corollary}

\newcommand{\ii}{\infty}
\newcommand\R{{\ensuremath {\mathbb R} }}
\newcommand\C{{\ensuremath {\mathbb C} }}

\newcommand\1{{\ensuremath {\mathds 1} }}
\renewcommand\phi{\varphi}

\newcommand{\gS}{\mathfrak{S}}

\newcommand{\cW}{\mathcal{W}}
\newcommand{\cF}{\mathcal{F}}

\renewcommand{\epsilon}{\varepsilon}
\newcommand\pscal[1]{{\ensuremath{\left\langle #1 \right\rangle}}}
\newcommand{\norm}[1]{ \left| \! \left| #1 \right| \! \right| }
\newcommand{\tr}{{\rm Tr}}

\renewcommand{\geq}{\geqslant}
\renewcommand{\leq}{\leqslant}

\title[Strichartz inequality for orthonormal functions]{Strichartz inequality for orthonormal functions}

\author[R.L. Frank]{Rupert L. FRANK}
\address{Department of Mathematics, Caltech, Pasadena, CA 91125, USA}
\email{rlfrank@caltech.edu}

\author[M. Lewin]{Mathieu LEWIN}
\address{CNRS \& Universit\'e de Cergy-Pontoise, Mathematics Department (UMR 8088), F-95000 Cergy-Pontoise, France} 
\email{mathieu.lewin@math.cnrs.fr}

\author[E.H. Lieb]{Elliott H. LIEB}
\address{Departments of Mathematics and Physics, Princeton University, Jadwin Hall, Washington Road, Princeton, NJ
08544, USA}
\email{lieb@princeton.edu}

\author[R. Seiringer]{Robert SEIRINGER}
\address{Institute of Science and Technology Austria, Am Campus 1, 3400 Klosterneuburg, Austria}
\email{robert.seiringer@ist.ac.at}

\date{\today}

\begin{document}

\begin{abstract}
We prove a Strichartz inequality for a system of orthonormal functions, with an optimal behavior of the constant in the limit of a large number of functions. The estimate generalizes the usual Strichartz inequality, in the same fashion as the Lieb-Thirring inequality generalizes the Sobolev inequality. As an application, we consider the Schrödinger equation in a time-dependent potential and we show the existence of the wave operator in Schatten spaces.
\end{abstract}

\thanks{\copyright\,2013 by the authors. This paper may be reproduced, in its entirety, for non-commercial purposes. To appear in \emph{Journal of the European Mathematical Society}.}

\maketitle

\tableofcontents

\section{Introduction}

In quantum mechanics, a system of $N$ independent fermions in $\R^d$ is described by a collection of $N$ orthonormal functions $u_1,...,u_N$ in $L^2(\R^d)$. For this reason, functional inequalities involving a large number of orthonormal functions are very useful in the mathematical analysis of large quantum systems. In~\cite{LieThi-75,LieThi-76}, Lieb and Thirring proved the first bound of this kind:
\begin{equation}
\int_{\R^d}\left(\sum_{j=1}^N |\nabla u_j(x)|^2\right)\, dx\geq C\int_{\R^d}\left(\sum_{j=1}^N|u_j(x)|^2\right)^{1+\tfrac{2}{d}}dx
\label{eq:Lieb-Thirring}
\end{equation}
where $C>0$ is independent of $N$ and of the orthonormal functions $u_j$. The Lieb-Thirring inequality~\eqref{eq:Lieb-Thirring} generalizes the Gagliardo-Nirenberg-Sobolev inequality
\begin{equation}
\int_{\R^d}|\nabla u(x)|^2\, dx\geq C'\int_{\R^d}|u(x)|^{2+\tfrac{4}{d}}dx
\label{eq:Sobolev}
\end{equation}
for an $L^2$--normalized function $u$ and it is a fundamental tool for understanding the stability of matter~\cite{Lieb-76,Lieb-90,LieSei-09}. The orthogonality between the functions $u_j$ is essential here to get the bound~\eqref{eq:Lieb-Thirring}. Using the Sobolev inequality~\eqref{eq:Sobolev} and the triangle inequality, we would only obtain a constant $C$ that goes to $0$ in the limit $N\to\ii$. For other inequalities for systems of orthonormal functions, see, for example,~\cite{Lieb-83d}.

The purpose of this article is to prove a generalization of the well known Strichartz inequality for systems of orthonormal functions. We expect that our new inequality will play an important role in understanding dispersive effects in large or infinite quantum systems. 

\section{An inequality for orthonormal functions and its dual}

\subsection{Strichartz inequality for orthonormal functions}
We recall that, in the case of the Schrödinger equation, the Strichartz inequality reads
\begin{equation}
\int_\R  \left(\int_{\R^d}\left|\big(e^{it\Delta}u\big)(x) \right|^{2q}\,dx\right)^{\tfrac{p}{q}}\,dt\leq C\left(\int_{\R^d}|u(x)|^2\,dx\right)^{p}
\label{eq:usual-Strichartz}
\end{equation}
where $p,q\geq1$ satisfy $(p,q,d)\neq (1,\ii,2)$ and
\begin{equation}
\frac{2}{p}+\frac{d}{q}=d, 
\label{eq:condition-p-q}
\end{equation}
see~\cite{Strichartz-77,Yajima-87,GinVel-92,GinVel-95,LinSog-95,KeeTao-98,Cazenave-03,Tao-06}. Here $e^{it\Delta}u$ is the unique solution to the free Schrödinger equation $i\,\dot{u}(t,x)=-\Delta u(t,x)$ such that $u(0,x)=u(x)$.
Our main result is the following

\begin{theorem}[Strichartz inequality for orthonormal functions]\label{thm:version_u}
Assume that $p,q,d\geq1$ satisfy
$$1< q\leq 1+\frac{2}{d}\quad\text{and}\quad \frac{2}{p}+\frac{d}{q}=d.$$
For any (possibly infinite) system $(u_j)$ of orthonormal functions in $L^2(\R^d)$ and any coefficients $(n_j)\subset\C$, we have 
\begin{equation}
\boxed{\int_\R  \left(\int_{\R^d} \bigg|\sum_{j}n_j\left|\big(e^{it\Delta}u_j\big)(x) \right|^{2}\bigg|^q\,dx\right)^{\tfrac{p}{q}}\!dt\leq C_{d,q}^p\left(\sum_{j}|n_j|^{\tfrac{2q}{q+1}}\right)^{\tfrac{p(q+1)}{2q}}}
\label{eq:Strichartz-orth-fns}
\end{equation}
where $C_{d,q}$ is a universal constant which only depends on $d$ and $q$.
\end{theorem}

\begin{remark}
For $q=1$ and $p=\ii$, we have the bound 
\begin{equation}
\sup_{t\in\R} \left(\int_{\R^d}\bigg|\sum_{j}n_j\left|\big(e^{it\Delta}u_j\big)(x) \right|^{2}\bigg|\,dx\right)\leq \sum_{j}|n_j|
\label{eq:q=1}
\end{equation}
which is an obvious consequence of the triangle inequality and of the fact that $e^{it\Delta}$ is a unitary operator on $L^2(\R^d)$, for any fixed $t\in\R$. Note that~\eqref{eq:q=1} does not use the orthogonality of the functions $u_j$.
\end{remark}

The inequality~\eqref{eq:Strichartz-orth-fns} can be rewritten in a convenient form in terms of the operator 
$$\gamma:=\sum_{j} n_j|u_j\rangle\langle u_j|$$
which acts on $L^2(\R^d)$. Here we have used Dirac's notation $|u\rangle\langle v|$ for the rank-one operator $f\mapsto \pscal{v,f}u$. Because the $u_j$ form an orthonormal system, the $n_j$ are precisely the eigenvalues of the operator $\gamma$. The evolved operator 
$$\gamma(t):=e^{it\Delta}\gamma e^{-it\Delta}=\sum_{j} n_j|e^{it\Delta}u_j\rangle\langle e^{it\Delta}u_j|$$
solves (in a weak sense) the von-Neumann Schrödinger equation
$i\dot\gamma(t)=[-\Delta,\gamma(t)]$ with $\gamma(0)=\gamma$. Introducing the density $\rho_{\gamma(t)}:=\sum_{j} n_j|e^{it\Delta}u_j|^2$
we see that~\eqref{eq:Strichartz-orth-fns} can be reformulated as
\begin{equation}
\boxed{\norm{\rho_{\gamma(t)}}_{L^p_t(\R,L^q_x(\R^d))}\leq C_{d,q}\norm{\gamma}_{\gS^{\frac{2q}{q+1}}}}
\label{eq:Strichartz-Schatten}
\end{equation}
where
\begin{equation*}
\norm{\gamma}_{\gS^{\frac{2q}{q+1}}}:=\left(\sum_{j}|n_j|^{\frac{2q}{q+1}}\right)^{\frac{q+1}{2q}}
\end{equation*}
is called the \emph{Schatten norm} of the operator $\gamma$ (see for instance~\cite{Simon-79} for elementary properties of Schatten spaces). The main advantage of the formulation~\eqref{eq:Strichartz-Schatten} is that we do not need to specify the functions $u_j$ and the complex numbers $n_j$ anymore, they are now all included in the operator $\gamma$.

The coefficients $n_j$ need not be real. In practice, the operator $\gamma$ is the one particle density matrix of fermions and it must satisfy the Pauli principle $0\leq\gamma\leq1$, which means that $0\leq n_j\leq 1$ for all $j$. Of particular interest is the case of $\gamma$ being a finite-rank orthogonal projection, that is, when $N$ of the $n_j$ are equal to 1 and the others vanish:
\begin{equation}
\int_\R  \left(\int_{\R^d} \bigg(\sum_{j=1}^N \left|\big(e^{it\Delta}u_j\big)(x) \right|^{2}\bigg)^q\,dx\right)^{\tfrac{p}{q}}dt\leq C_{d,q}^p N^{\tfrac{p(q+1)}{2q}}.
\label{eq:Strichartz-orth-fns-N}
\end{equation}
The inequality~\eqref{eq:Strichartz-orth-fns-N} has a much better scaling with respect to the number $N$ of functions, than the power $N^p$ which can be obtained by using the usual Strichartz inequality~\eqref{eq:usual-Strichartz} and the triangle inequality (as was stated before in~\cite{Castella-97}, for example). The power is decreased to $p(q+1)/(2q)<p$ due to the orthonormality condition.

\subsection{Optimality of the Schatten exponent}

Using a semi-classical argument based on coherent states, we can prove that the power $p(q+1)/(2q)$ in~\eqref{eq:Strichartz-orth-fns} and in~\eqref{eq:Strichartz-orth-fns-N} is optimal, it cannot be decreased further. This is the content of the following

\begin{proposition}[Optimality of the Schatten exponent]\label{prop:semi-classical}
Assume that $d,p,q\geq1$ satisfy $2/p+d/q=d$. Then we have 
\begin{equation}
\sup_{\gamma\in \gS^{r}} \frac{\norm{\rho_{e^{it\Delta }\gamma e^{-it\Delta }}}_{L^p_t(\R,L^q_x(\R^d))}}{\norm{\gamma}_{\gS^r}}=+\ii
\label{eq:optimal-Schatten}
\end{equation}
for all $r>\frac{2q}{q+1}$.
\end{proposition}

We now present a heuristic computation explaining why Theorem~\ref{thm:version_u} can be thought of as a semi-classical bound.
In a certain sense this is also the idea behind the proof of Proposition~\ref{prop:semi-classical} (given in Section~\ref{sec:proof_prop_semi-classics} below).

We consider a system of fermions which, at time $t=0$, occupy a cube of side length $L$. We assume that the electron density on the cube at time $t=0$ is a constant $\rho>0$ (and zero outside this cube). Thus, the total number of particles is $N\sim\rho L^d$.

As $|t|$ increases the electrons disperse and, after a certain time $T$ we consider them as roughly having disjoint supports. For $|t|\geq T$ we can apply the ordinary Strichartz inequality and, because of the disjoint support condition, the left side of~\eqref{eq:Strichartz-orth-fns-N} is of the order of $N^{p/q}$ (by the triangle inequality for the $t$-integration and the fact that $p\geq q$), which is much smaller than what we try to prove. Thus, it remains to compute the order of magnitude of $T$.

We think of $T$ as the typical time it takes an electron to move a distance comparable with the size of the system. Thomas--Fermi theory says that the momentum $p$ per particle is $|p| \sim \rho^{1/d}$. If we assume that the electrons move ballistically, then $T \sim L/|p| \sim L \rho^{-1/d}$. 

Thus, the left side of (5), restricted to times $|t|\leq T$, is $\sim T (L^d \rho^q )^{p/q} \sim L \rho^{-1/d} L^{dp/q}\rho^p$. Because of the scaling condition $2/p+d/q=d$, this coincides with the value $N^{p(q+1)/(2q)}$ of the right side of~\eqref{eq:Strichartz-orth-fns-N}.

\subsection{Dual Strichartz inequality}
In this paper we will not provide a direct proof of Theorem~\ref{thm:version_u}, but we will rather prove an inequality that is \emph{dual} to~\eqref{eq:Strichartz-orth-fns} and which we describe in this section. It is an interesting open problem to provide a direct proof of~\eqref{eq:Strichartz-orth-fns}. For the Lieb-Thirring inequality~\eqref{eq:Lieb-Thirring}, this has been solved only recently by Rumin in~\cite{Rumin-10}.

We recall that for any (locally) trace-class operator $\gamma$ and any bounded function $V$ of compact support,
$$\tr(V(x)\gamma)=\int_{\R^d} V(x)\rho_\gamma(x)\,dx,$$
where $V(x)$ on the left is identified with the corresponding multiplication operator on $L^2(\R^d)$. 
For a time-dependent potential $V(t,x)\in L^\ii_c(\R\times\R^d)$, we therefore obtain
\begin{align*}
\left|\tr\left( \int_\R  e^{-it\Delta}V(t,x)e^{it\Delta}\,dt\right)\gamma\right|&=\left|\int_\R  \tr(Ve^{it\Delta}\gamma e^{-it\Delta})\,dt\right|\\
&=\left|\int_\R \int_{\R^d} V(t,x)\rho_{\gamma(t)}(x)\,dx\,dt\right|\\
&\leq\norm{V}_{L^{p'}_t(\R,L^{q'}_x(\R^d))}\norm{\rho_{\gamma(t)}}_{L^{p}_t(\R,L^{q}_x(\R^d))}
\end{align*}
where $p'$ and $q'$ are the exponents dual to $p$ and $q$. Hence, by duality Theorem~\ref{thm:version_u} turns out to be equivalent to the following 

\begin{theorem}[Strichartz inequality in Schatten spaces, dual version]\label{thm:version_V}
Assume that $p',q',d\geq1$ satisfy
$$ 1+\frac{d}{2}\leq q'<\ii\quad\text{and}\quad \frac{2}{p'}+\frac{d}{q'}=2.$$
We have
\begin{equation}
\boxed{\norm{\int_\R  e^{-it\Delta}V(t,x)e^{it\Delta}\,dt}_{\gS^{2q'}}\leq C_{d,q}\norm{V}_{L^{p'}_t(\R,L^{q'}_x(\R^d))}}
\label{eq:Strichartz-V}
\end{equation}
where $C_{d,q}$ is the same constant as in Theorem~\ref{thm:version_u}.
\end{theorem}

\begin{remark}
For $q'=\ii$ and $p'=1$, we have the bound
\begin{equation}
\norm{\int_\R  e^{-it\Delta}V(t,x)e^{it\Delta}\,dt}\leq \norm{V}_{L^{1}_t(\R,L^{\ii}_x(\R^d))}.
\label{eq:Strichartz-V2}
\end{equation}
\end{remark}

The dual version of the usual Strichartz inequality~\eqref{eq:usual-Strichartz} is
\begin{equation}
\norm{\int_\R  e^{-it\Delta}V(t,x)e^{it\Delta}\,dt}\leq C\norm{V}_{L^{p'}_t(\R,L^{q'}_x(\R^d))}.
\label{eq:usual-Strichartz-dual}
\end{equation}
The replacement of the operator norm on the left by the Schatten norm $\gS^{2q'}$ for $q'<\ii$ is our main contribution. Of course, since the Schatten spaces form an increasing sequence, we deduce that 
\begin{equation}
\norm{\int_\R  e^{-it\Delta}V(t,x)e^{it\Delta}\,dt}_{\gS^{r}}\leq C\norm{V}_{L^{p'}_t(\R,L^{q'}_x(\R^d))},\qquad \forall r\geq 2q'.
\label{eq:Strichartz-V-all-r}
\end{equation}

\medskip

Using~\eqref{eq:Strichartz-V}, we are also able to prove an inhomogeneous inequality. Consider the equation
\begin{equation}
\begin{cases}
i\dot\gamma(t)=[-\Delta,\gamma(t)]+iR(t)\\[0.2cm]
\gamma(t_0)=0
\end{cases}
\label{eq:Strichartz-inhomogeneous}
\end{equation}
where $R(t)$ is a self-adjoint operator on $L^2(\R^d)$ which, say, is bounded for almost every $t$. The solution can be written as 
\begin{equation}
\gamma(t)=\int_{t_0}^t e^{i(t-s)\Delta}R(s)e^{i(s-t)\Delta}\,ds.
\label{eq:gamma_inhomogeneous}
\end{equation}

\begin{corollary}[Inhomogeneous Strichartz inequality]\label{cor:Strichartz_inhomogeneous}
Assume that $p,q,d\geq1$ satisfy
$$1< q\leq 1+\frac{2}{d}\quad\text{and}\quad \frac{2}{p}+\frac{d}{q}=d$$
and let $\gamma(t)$ be given by~\eqref{eq:gamma_inhomogeneous}. Then we have
\begin{equation}
\norm{\rho_{\gamma(t)}}_{L^{p}_t(\R,L^q_x(\R^d))}\leq C\norm{\int_\R e^{-is\Delta}|R(s)|e^{is\Delta}\,ds}_{\gS^{\frac{2q}{q+1}}}
\label{eq:Strichartz_inhomogeneous}
\end{equation}
for a constant $C$ which is independent of $t_0$.
\end{corollary}

The proof of Corollary~\ref{cor:Strichartz_inhomogeneous} is again based on a dual argument. The idea is to write
\begin{align*}
&\left| \int_{t_0}^\ii  \int_{\R^d} V(t,x)\rho_{\gamma(t)}(x)\,dx\,dt\right|\\
&\qquad\qquad=\left|\int_{t_0}^\ii  \tr\big(V(t,x)\gamma(t)\big)\,dt\right|\\
&\qquad\qquad=\left|\int_{t_0}^\ii \int_{t_0}^t \tr\big(e^{-it\Delta}V(t,x)e^{it\Delta}e^{-is\Delta}R(s)e^{is\Delta}\big)\,ds\,dt\right|\\
&\qquad\qquad\leq\int_{t_0}^\ii \int_{t_0}^t \tr\big(e^{-it\Delta}|V(t,x)|e^{it\Delta}e^{-is\Delta}|R(s)|e^{is\Delta}\big)\,ds\,dt\\
&\qquad\qquad\leq \tr\left(\left(\int_{t_0}^\ii e^{-it\Delta}|V(t,x)|e^{it\Delta}\,dt\right)\left(\int_{t_0}^\ii e^{-is\Delta}|R(s)|e^{is\Delta}\,ds\right)\right).
\end{align*}
In the first inequality we have used that $|\tr(AB)|\leq \tr(|A|\,|B|)$ for all self-adjoint operators $A$ and $B$. It then remains to use Hölder's inequality for traces and~\eqref{eq:Strichartz-V} for the term involving $V(t,x)$. The argument is the same for times $t\leq t_0$.

\subsection{The end point}\label{sec:end-point}

We believe that our Strichartz inequality~\eqref{eq:Strichartz-orth-fns} is true for all 
\begin{equation}
1\leq q<\frac{d+1}{d-1}
\label{eq:optimal}
\end{equation}
but so far we are missing the result in the interval $1+2/d<q<(d+1)/(d-1)$. This corresponds to the range $(d+1)/2<q'<1+d/2$ for the dual inequality~\eqref{eq:Strichartz-V}.

We can prove that the operator $\int_\R e^{-it\Delta}V(t,x)e^{it\Delta}\,dt$ is never in the Schatten space $\gS^{d+1}$, even when $V$ has a fast decay in space and time. This means that the Strichartz inequality~\eqref{eq:Strichartz-V} cannot hold at $p'=d+1$ and $q'=(d+1)/2$, and that the condition~\eqref{eq:optimal} is necessary. 

\begin{proposition}[The end point]\label{prop:end-point}
Let $0\neq V\in L^\ii_c(\R\times\R^d)$ be a non-negative function with non-negative Fourier transform (in both space and time). Then
\begin{equation}
\tr\left(\int_\R e^{-it\Delta}V(t,x)e^{it\Delta}\,dt\right)^{d+1}=+\ii.
\end{equation}
\end{proposition}

We find logarithmically divergent integrals at $(p',q')=(d+1,(d+1)/2)$, which suggests that the operator $\int_\R e^{-it\Delta}V(t,x)e^{it\Delta}\,dt$ is in the weak Schatten space $\gS_{\rm w}^{d+1}$ when $V\in L^{d+1}_t(\R,L_x^{(d+1)/2}(\R^d))$. If true, this would imply the bound
$$\norm{\rho_{\gamma(t)}}_{L^{1+1/d}_t(\R,L^{(d+1)/(d-1)}_x(\R^d))}\leq C\norm{\gamma}_{\gS^{r}},\quad\forall 1\leq r<1+\frac{1}{d}.$$
This estimate would follow from the bound~\eqref{eq:Strichartz-Schatten} if it were true at the end point $(p,q)=(1+1/d,(d+1)/(d-1))$, and it is therefore weaker than~\eqref{eq:Strichartz-Schatten}.

\section{Application: the Schrödinger wave operator for time-dependent potentials}\label{sec:wave}

In this section we consider the wave operator for a time-dependent potential $V(t,x)$. Using our previous estimates we will be able to define it in Schatten spaces.

Let $V(t,x)\in L^{p'}_t(\R,L^{q'}_x(\R^d))$ with $p'$ and $q'$ as in Theorem~\ref{thm:version_V}. We consider the time-dependent Schrödinger equation
\begin{equation}
\left\{
\begin{array}{rcl}
\displaystyle i\,\frac{\partial}{\partial t}u(t,x) &=& \big(-\Delta+V(t,x)\big)u(t,x)\\[0.2cm]
u(t_0,x)&=&u_0(x)
\end{array}\right.
\label{eq:TD-Schrodinger} 
\end{equation}
and we define the associated unitary propagator $U_V(t,t_0)$, which is such that 
$$
\left\{
\begin{array}{rcl}
\displaystyle i\frac{\partial}{\partial t}U_V(t,t_0)&=&\big(-\Delta+V(t,x)\big)U_V(t,t_0)\\[0.3cm]
U_V(t_0,t_0)&=&1.
\end{array}\right.
$$
Therefore, the unique solution to the time-dependent Schrödinger equation~\eqref{eq:TD-Schrodinger} can be written $u(t)=U_V(t,t_0)u_0$. The proof that $U_V(t,t_0)$ is well-defined under our assumptions on $V$ can be found for instance in~\cite{Yajima-87}.

The wave operator is defined by
\begin{equation}
\cW_V(t,t_0):=U_0(t_0,t)U_V(t,t_0)=e^{i(t_0-t)\Delta}U_V(t,t_0) 
\end{equation}
and it solves the equation in the ``interaction picture''
\begin{equation}
\left\{
\begin{array}{rcl}
\displaystyle i\frac{\partial}{\partial t}\cW_V(t,t_0)&=&e^{i(t_0-t)\Delta}V(t,x)e^{i(t-t_0)\Delta}\; \cW_V(t,t_0)\\
\cW_V(t_0,t_0)&=&1.
\end{array}\right.
\label{eq:interaction_picture} 
\end{equation}
The unique solution to~\eqref{eq:interaction_picture} can be written in a Dyson series as
\begin{equation}
\cW_V(t,t_0)=1+\sum_{n\geq1}\cW_V^{(n)}(t,t_0),
\label{eq:Dyson_series} 
\end{equation}
where the $n$th order is the operator
\begin{multline}
\cW_V^{(n)}(t,t_0):=(-i)^n \int_{t_0}^{t}dt_n\int_{t_0}^{t_n}dt_{n-1}\cdots \int_{t_0}^{t_{2}}dt_1\;e^{i(t_0-t_n)\Delta}\\ V\big(t_{n},x\big)e^{i(t_n-t_{n-1})\Delta}\cdots e^{i(t_2-t_1)\Delta}V\big(t_1,x\big)e^{i(t_1-t_0)\Delta}.
\end{multline}
Note, in particular, that the first order is
$$\cW_V^{(1)}(t,t_0):=-i\int_{t_0}^{t}e^{i(t_0-s)\Delta}V\big(s,x\big)e^{i(s-t_0)\Delta}\,ds,$$
which we have already estimated in Theorem~\ref{thm:version_V}. It admits a limit as $t\to\pm\ii$ in the Schatten space $\gS^{2q'}$ when $V\in L^{p'}_t(\R,L^{q'}_x(\R^d))$. 

Taking $t\to\pm\ii$ in all the terms leads to the (formal) wave operator
\begin{equation}
\cW_{V,\pm}(t_0)=1+\sum_{n\geq1}\cW_{V,\pm}^{(n)}(t_0)
\label{eq:Dyson_series2} 
\end{equation}
with 
\begin{multline}
\cW_{V,\pm}^{(n)}(t_0):=(-i)^n \int_{t_0}^{\pm\ii}dt_n\int_{t_0}^{t_n}dt_{n-1}\cdots \int_{t_0}^{t_{2}}dt_1\;e^{i(t_0-t_n)\Delta}\\ V\big(t_{n},x\big)e^{i(t_n-t_{n-1})\Delta}\cdots e^{i(t_2-t_1)\Delta}V\big(t_1,x\big)e^{i(t_1-t_0)\Delta}.
\end{multline}
The finite-time wave operators $\cW^{(n)}_V(t,t_0)$ can be recovered by taking a potential $V$ of compact support in time. The series~\eqref{eq:Dyson_series2} defining $\cW_{V,\pm}$ is known to converge in the operator norm when $V\in L^1_t(L^\ii_x)$, see~\cite[Prop. 2.2]{Jensen-94}. Simply, we have in this case
\begin{align}
\norm{\cW_{V,+}^{(n)}(t_0)}&\leq \int_{t_0}^{\ii}dt_n\int_{t_0}^{t_n}dt_{n-1}\cdots \int_{t_0}^{t_{2}}dt_1 \norm{V\big(t_{n},\cdot)}_{L^\ii_x}\cdots\norm{V\big(t_1,\cdot)}_{L^\ii_x}\nonumber\\
&=\frac{1}{n!}\norm{V}_{L^1(t_0,\ii;L^\ii_x)}^n,\label{eq:estim_wave_simple}
\end{align}
which proves that the series~\eqref{eq:Dyson_series2} has an infinite radius of convergence. In particular there is no size condition on $\norm{V}_{L^1_t(L^\ii_x)}$. The argument is the same for $\cW_{V,-}^{(n)}(t_0)$. The existence of the wave operators for time-dependent potentials $V\in L^{p'}_t(L^{q'}_x)$ has been discussed in several works, including for instance~\cite{Howland-74,Yajima-87,Yajima-91,Jensen-94,Jensen-98,RodSch-04,AncPieVis-05,Pierfelice-06,NaiSte-06}. 

Our main result is a control of the Schatten norm of $\cW_{V,\pm}^{(n)}(t_0)$ in terms of the $L^{p'}_t(L^{q'}_x)$ norm of the potential $V$, which generalizes the operator norm bound~\eqref{eq:estim_wave_simple}. It makes the series~\eqref{eq:Dyson_series} convergent in Schatten spaces, independently of the size of the norm of $V$.

\begin{theorem}[Wave operator in Schatten spaces]\label{thm:wave-op}
For $V\in L^{p'}_t(\R,L^{q'}_x(\R^d))$ with $p'$ and $q'$ as in Theorem~\ref{thm:version_V}, we have
\begin{equation}
\boxed{\norm{\cW_{V,\pm}^{(n)}(t_0)}_{\gS^{2\left\lceil\frac{q'}{n}\right\rceil}}\leq \frac{C^n}{(n!)^{\frac{1}{p'}-\epsilon}}\norm{V}^n_{L^{p'}_t(\R,L^{q'}_x(\R^d))}}
\end{equation}
for every $\varepsilon>0$, $n\geq2$, $t_0\in\R$, and some constant $C$ which only depends on $\epsilon,d,q'$. In particular, the map 
$$V\in L^{p'}_t(\R,L^{q'}_x(\R^d))\mapsto \cW_{V,\pm}(t_0)-1\in\gS^{2q'}$$
is a smooth function.
\end{theorem}

Since $2\lceil{q'}/{n}\rceil\leq 2q'$ for all $q'\geq 1+d/2$ and all $n\geq2$, it follows from Theorems~\ref{thm:version_V} and~\ref{thm:wave-op} that the scattering matrix
$$S_V(t_0)=\cW_{V,+}(t_0)\,\cW_{V,-}(t_0)^*$$
belongs to $1+\gS^{2q'}$, under our assumptions on the time-dependent potential $V(t,x)$. It is a smooth function of $V$ in the space  $L^{p'}_t(\R,L^{q'}_x(\R^d))$.

\section{Proofs}

\subsection{Proof of Theorems~\ref{thm:version_u} and~\ref{thm:version_V}: the main inequality}

The duality argument showing that Theorem~\ref{thm:version_u} is equivalent to Theorem~\ref{thm:version_V} has already been sketched before and we leave the details to the reader. All the manipulations can be justified by assuming first that $V\in L^\ii_c(\R\times\R^d)$ and that $\gamma$ is finite rank.

We have to show that the operator 
$$V\in L^{p'}_t(\R,L^{q'}_x(\R^d))\mapsto \int_\R  e^{-it\Delta} V(t,x) e^{it\Delta}\,dt\in \gS^{2q'}$$
is bounded. By the complex interpolation method~\cite[Chap. 4]{BerLof-76}, it is sufficient to prove this fact at the two points $(p',q')=(1,\ii)$ and $(p',q')=(1+d/2,1+d/2)$.

For $(p',q')=(1,\ii)$, the argument is well-known. We simply bound the operator norm by
\begin{align*}
\norm{\int_\R  e^{-it\Delta} V(t,x) e^{it\Delta}\,dt}&\leq \int_\R  \norm{e^{-it\Delta} V(t,x) e^{it\Delta}}\,dt\\
&=\int_\R  \norm{V(t,\cdot)}_{L^\ii(\R^d)}\,dt 
\end{align*}
which is the desired estimate.

Let us turn to the case $p'=q'=1+d/2$. Without any loss of generality, we may assume that $V\geq0$. We then have $e^{-it\Delta} V(t,x) e^{it\Delta}\geq0$ as an operator on $L^2(\R^d)$, for all $t\in\R$. We can also assume that $V\in L^\ii_c(\R\times \R^d)$ (the final estimate follows from a monotone convergence argument).

It will be useful to shorten our notation. First we recall that
$$e^{-it\Delta} x e^{it\Delta}=x-2it\nabla$$
where $x$ is here identified with the multiplication operator by $x$, and which can be seen by differentiating with respect to $t$. By the functional calculus we deduce that
$$e^{-it\Delta} f(x) e^{it\Delta}=f(x+2tp)$$
with $p:=-i\nabla$. From this we deduce that
$$e^{-it\Delta}V(t,x)e^{it\Delta}=V(t,x+2tp).$$
Using that $V\geq0$, we can write the Schatten norm as
\begin{align}
&\norm{\int_\R  e^{-it\Delta} V(t,x) e^{it\Delta}\,dt}^{d+2}_{\gS^{d+2}}\nonumber\\
&\qquad=\tr\left(\int_\R  e^{-it\Delta} V(t,x) e^{it\Delta}\,dt\right)^{d+2}\nonumber\\
&\qquad=\tr\left(\int_\R  V(t,x+2tp)\,dt\right)^{d+2}\nonumber\\
&\qquad=\tr\left(\int_\R  \cdots  \int_\R V(t_1,x+2t_1p)\cdots V(t_{d+2},x+2t_{d+2}p)dt_1\cdots dt_{d+2}\right).\label{eq:first-computation}
\end{align}
The first step is to exchange the trace and the integral and, in order to justify this manipulation, we need to prove that
\begin{equation}
\int_\R  \cdots  \int_\R \norm{V(t_1,x+2t_1p)\cdots V(t_{d+2},x+2t_{d+2}p)}_{\gS^1}dt_1\cdots dt_{d+2}<\ii
\label{eq:integrable_trace_norm}
\end{equation}
(at least for $V\in L^\ii_c(\R\times\R^d)$, which we assume throughout here). In order to estimate the trace norm in the integral, we make use of the following 

\begin{lemma}\label{lem:gKSS}
Let $\alpha,\beta,\gamma,\delta\in\R$. We have 
\begin{equation}
\norm{f(\alpha x+\beta p)\, g(\gamma x+\delta p)}_{\gS^r}\leq \frac{\norm{f}_{L^r(\R^d)}\norm{g}_{L^r(\R^d)}}{(2\pi)^{\frac{d}{r}}|\alpha\delta-\beta\gamma|^{\frac{d}{r}}}
\label{eq:gKSS}
\end{equation}
for all $r\geq2$.
\end{lemma}

For $\alpha=\delta=1$ and $\beta=\gamma=0$, the estimate is just the well-known Kato-Seiler-Simon inequality
\begin{equation}
\norm{f(x)\, g(p)}_{\gS^r}\leq (2\pi)^{-d/r} \norm{f}_{L^r(\R^d)}\norm{g}_{L^r(\R^d)},
\end{equation}
see~\cite{SeiSim-75} and~\cite[Thm 4.1]{Simon-79}. The generalization~\eqref{eq:gKSS} implicitly appears in~\cite[Sec. 2.1]{BirKarSol-91}.

We postpone the proof of the lemma and go on with the proof of~\eqref{eq:integrable_trace_norm}. Using the fact that $V\geq0$ and Hölder's inequality in Schatten spaces, we write
\begin{align*}
&\norm{V(t_1,x+2t_1p)\cdots V(t_{d+2},x+2t_{d+2}p)}_{\gS^1}\nonumber\\
&\qquad\quad = \Big\|V(t_1,x+2t_1p)\sqrt{V(t_2,x+2t_2p)}\sqrt{V(t_2,x+2t_2p)}\cdots\\
&\qquad\qquad\qquad \cdots  \sqrt{V(t_{d+1},x+2t_{d+1}p)} V(t_{d+2},x+2t_{d+2}p)\Big\|_{\gS^1}\nonumber\\
&\quad\qquad \leq \norm{V(t_1,x+2t_1p)\sqrt{V(t_2,x+2t_2p)}}_{\gS^{d+1}}\times\\
&\qquad\qquad \times\norm{\sqrt{V(t_2,x+2t_2p)}\sqrt{V(t_3,x+2t_3p)}}_{\gS^{d+1}}\times \cdots \nonumber\\
&\qquad\qquad \cdots \times \norm{\sqrt{V(t_{d+1},x+2t_{d+1}p)}V(t_{d+2},x+2t_{d+2}p)}_{\gS^{d+1}}.
\end{align*}
Using now~\eqref{eq:gKSS} and the fact that $V\in L^\ii_c(\R\times\R^d)$, we get
\begin{align}
&\norm{V(t_1,x+2t_1p)\cdots V(t_{d+2},x+2t_{d+2}p)}_{\gS^1}\nonumber\\
&\qquad\leq \frac{\norm{V(t_1,\cdot)}_{L^{d+1}_x}\norm{V(t_2,\cdot)}_{L^{\frac{d+1}2}_x}\cdots \norm{V(t_{d+1},\cdot)}_{L^{\frac{d+1}2}_x}\norm{V(t_{d+2},\cdot)}_{L^{d+1}_x}}{(4\pi)^{d}|t_1-t_2|^{\frac{d}{d+1}}\cdots |t_{d+1}-t_{d+2}|^{\frac{d}{d+1}}}\nonumber\\
&\qquad\leq C\frac{\prod_{j=1}^{d+2}\1(a\leq t_j\leq b)}{(4\pi)^{d}|t_1-t_2|^{\frac{d}{d+1}}\cdots |t_{d+1}-t_{d+2}|^{\frac{d}{d+1}}}\label{eq:estim_trace-norm}
\end{align}
where $(a,b)$ is the support of $V$ in the time variable. At this step we use the multilinear Hardy-Littlewood-Sobolev inequality.

\begin{theorem}[Multilinear Hardy-Littlewood-Sobolev inequality]
Assume that $(\beta_{ij})_{1\leq i,j\leq N}$ and $(r_k)_{1\leq k\leq N}$ are real-numbers such that 
\begin{equation}
\beta_{ii}=0,\quad 0\leq \beta_{ij}=\beta_{ji}<1,\quad r_k>1,\quad \sum_{k=1}^N\frac1{r_k}>1,\quad \sum_{i=1}^N\beta_{ik}=\frac{2(r_k-1)}{r_k}.
\label{eq:conditions_multilinear_HLS}
\end{equation}
Then there exists a constant $C$ such that
\begin{equation}
\left|\int_\R  \cdots \int_\R \frac{f_1(t_1)\cdots f_N(t_{N})}{\prod_{i<j}|t_i-t_j|^{\beta_{ij}}}dt_1\cdots dt_{N}\right|
\leq C\prod_{k=1}^N\norm{f_k}_{L^{r_k}(\R)}
\label{eq:multilinear_HLS}
\end{equation}
for all $f_k\in L^{r_k}(\R)$.
\end{theorem}

The multilinear Hardy-Littlewood-Sobolev inequality can be found in~\cite[Thm. 6]{Beckner-95} and, in the particular case where all the $\beta_{ij}$ and the $r_k$ are identical, in~\cite[Prop. 2.2]{Christ-85}. Applying~\eqref{eq:multilinear_HLS} with $r_1=r_{d+2}=2(d+1)/(d+2)$ and $r_2=\cdots=r_{d+1}=d+1$, we conclude that~\eqref{eq:integrable_trace_norm} holds.
Hence we have shown that
\begin{multline*}
\tr\left(\int_\R  e^{-it\Delta} V(t,x) e^{it\Delta}\,dt\right)^{d+2}\\=\int_\R  \cdots  \int_\R \tr\Big(V(t_1,x+2t_1p)\cdots V(t_{d+2},x+2t_{d+2}p)\Big)dt_1\cdots dt_{d+2}.
\end{multline*}

By following the previous argument we will now derive a more symmetric estimate on the trace in the integral. Simply, we use the cyclicity of the trace and get, this time,
\begin{align}
&\left|\tr\Big( V(t_1,x+2t_1p)\cdots V(t_{d+2},x+2t_{d+2}p)\Big)\right|\nonumber\\
&\qquad = \left|\tr\Big(\sqrt{V(t_1,x+2t_1p)}\cdots V(t_{d+2},x+2t_{d+2}p)\sqrt{V(t_1,x+2t_1p)}\Big)\right|\nonumber\\
&\qquad \leq \norm{\sqrt{V(t_1,x+2t_1p)}\sqrt{V(t_2,x+2t_2p)}}_{\gS^{d+2}}\times \cdots \nonumber\\
&\qquad\qquad \cdots \times \norm{\sqrt{V(t_{d+1},x+2t_{d+1}p)}\sqrt{V(t_{d+2},x+2t_{d+2}p)}}_{\gS^{d+2}}\times\nonumber\\
 &\qquad\qquad\qquad \times \norm{\sqrt{V(t_{d+2},x+2t_{d+2}p)}\sqrt{V(t_1,x+2t_1p)}}_{\gS^{d+2}}.
\end{align}
With the aid of~\eqref{eq:gKSS} we obtain
\begin{multline*}
\left|\tr\Big( V(t_1,x+2t_1p)\cdots V(t_{d+2},x+2t_{d+2}p)\Big)\right|\\
\leq \frac{\norm{V(t_1,\cdot)}_{L^{1+d/2}_x}\cdots \norm{V(t_{d+2},\cdot)}_{L^{1+d/2}_x}}{(4\pi)^{d}|t_1-t_2|^{\frac{d}{d+2}}\cdots |t_{d+2}-t_1|^{\frac{d}{d+2}}}.
\end{multline*}
Using again the multilinear Hardy-Littlewood-Sobolev inequality~\eqref{eq:multilinear_HLS}, this time  with $r_1=\cdots =r_{d+2}=1+d/2$, we conclude that
\begin{equation*}
\tr\left(\int_\R  e^{-it\Delta} V(t,x) e^{it\Delta}\,dt\right)^{d+2}\\
\leq C\norm{V}_{L^{1+d/2}_{t,x}}^{d+2},
\end{equation*}
as we wanted.\qed 

\begin{remark}
At the end point $p'=d+1$, $q'=(d+1)/2$ we have precisely $d+1$ functions and the above proof cannot be applied, since $\sum_{k=1}^{d+1} 1/r_k=1$ in this case.
\end{remark}

\begin{remark}
It is useful to think of the semi-classical regime, in which
\begin{multline*}
\tr\Big( V(t_1,x+2t_1p)\cdots V(t_{d+2},x+2t_{d+2}p)\Big)\\
\simeq (2\pi)^{-d}\int_{\R^d}\int_{\R^d}V(t_1,x+2t_1p)\cdots V(t_{d+2},x+2t_{d+2}p)\,dx\,dp.
\end{multline*}
The right side can be estimated by $C_{\rm BL}\prod_{j=1}^N\norm{V(t_j,\cdot)}_{L_x^{1+d/2}(\R^d)}$. The best constant $C_{\rm BL}$ in this inequality was found by Brascamp and Lieb in~\cite{BraLie-76} and it can also be controlled by $|t_1-t_2|^{-d/(d+2)}\cdots |t_{d+2}-t_1|^{-d/(d+2)}$.
\end{remark}

It remains to provide the 

\begin{proof}[Proof of Lemma~\ref{lem:gKSS}]
For $r=\ii$ this is obvious. For $r=2$, we get
\begin{align*}
&\norm{f(\alpha x+\beta p)\, g(\gamma x+\delta p)}_{\gS^2}^2\\
&\quad=\tr \left[|f(\alpha x+\beta p)|^2\,|g(\gamma x+\delta p)|^2\right]\\
&\quad=\tr \left[e^{-i\frac{\beta}{2\alpha}p^2}|f(\alpha x)|^2e^{i\frac{\beta}{2\alpha}p^2}e^{-i\frac{\delta}{2\gamma}p^2}\,|g(\gamma x)|^2e^{i\frac{\delta}{2\gamma}p^2}\right]\\
&\quad=\tr \left[|f(\alpha x)|^2e^{i\frac{\beta\gamma-\alpha\delta}{2\alpha\gamma}p^2}|g(\gamma x)|^2e^{-i\frac{\beta\gamma-\alpha\delta}{2\alpha\gamma}p^2}\right]\\
&\quad=\left|\frac{\alpha\gamma}{2\pi(\beta\gamma-\alpha\delta)}\right|^{d}\iint |f(\alpha x)|^2e^{-i\frac{\alpha\gamma}{2(\beta\gamma-\alpha\delta)}|x-y|^2}|g(\gamma y)|^2e^{i\frac{\alpha\gamma}{2(\beta\gamma-\alpha\delta)}|y-x|^2}\,dx\,dy\\
&\quad=\frac{1}{(2\pi)^d|\beta\gamma-\alpha\delta|^d}\norm{f}_{L^2(\R^d)}^2\norm{g}_{L^2(\R^d)}^2.
\end{align*}
The inequality in $\gS^r$ now follows from complex interpolation or, alternatively, from the Lieb-Thirring inequality for matrices~\cite{LieThi-76}.
\end{proof}

\begin{remark}
By following Cwikel's proof (see~\cite{Cwikel-77} and \cite[Thm. 4.2]{Simon-79}), one can show the weak-type bound
\begin{equation}
\norm{f(\alpha x+\beta p)\, g(\gamma x+\delta p)}_{\gS^r_{\rm w}}\leq C_{d,r}\,\frac{\norm{f}_{L^r_{\rm w}(\R^d)}\norm{g}_{L^r(\R^d)}}{|\alpha\delta-\beta\gamma|^{\frac{d}{r}}}
\label{eq:gCwikel}
\end{equation}
for all $2<r<\ii$.
\end{remark}

\subsection{Proof of Proposition~\ref{prop:semi-classical}: optimality of the Schatten exponent}\label{sec:proof_prop_semi-classics}

Our proof is based on coherent states and ideas from semi-classical analysis. We will introduce a family of operators $\gamma$ depending on three positive parameters $\beta$, $L$ and $\mu$, which will be chosen appropriately at the end of the proof. To define $\gamma$ we use coherent states $F_{x,\xi}$, depending on parameters $x,\xi\in\R^d$,
$$
F_{x,\xi}(z) = (2\pi\beta)^{-d/4} e^{-(z-x)^2/(4\beta)} e^{i\xi\cdot z} \,.
$$
These functions are normalized in $L^2(\R^d)$ and satisfy
$$
\iint_{\R^d\times\R^d} \frac{dx\,d\xi}{(2\pi)^d} |F_{x,\xi}\rangle\langle F_{x,\xi}| = 1 \,.
$$
Now we define
$$
\gamma = \iint_{\R^d\times\R^d}  \frac{dx\,d\xi}{(2\pi)^d} e^{-x^2/L^2 - \xi^2/\mu} |F_{x,\xi}\rangle\langle F_{x,\xi}| \,.
$$
The relevant parameter in our computation is
$$
N = \int_{\R^d} \gamma(z,z) \,dz = \iint_{\R^d\times\R^d}  \frac{dx\,d\xi}{(2\pi)^d} e^{-x^2/L^2 - \xi^2/\mu} = A_d L^d \mu^{d/2}
$$
with an explicit constant $A_d$, depending only on $d$.

Obviously, $\gamma\geq 0$ and, by the Berezin--Lieb inequality~\cite{Berezin-72,Lieb-73b},
$$
\tr \gamma^r \leq \iint_{\R^d\times\R^d}  \frac{dx\,d\xi}{(2\pi)^d} e^{-rx^2/L^2 - r\xi^2/\mu} = r^{-d} N \,
$$
for $r\geq1$. Therefore the denominator in \eqref{eq:optimal-Schatten} does not exceed $r^{-d/r} N^{1/r}$, uniformly in $\beta$. To finish the proof we will now show that, by choosing $\beta,L,\mu$ appropriately, we can bound the numerator from below by a constant times $N^{(q+1)/(2q)}$. Then we choose $\mu$ and $L$ large and we get the result in the limit $N\to\infty$.

To carry out this strategy we compute the left side explicitly. We give the main steps of the computation. First,
$$
| e^{it\Delta} F_{x,\xi}(z)| = \left( \frac{\beta}{2\pi (\beta^2+t^2)}\right)^{d/4} e^{-\frac{\beta}{4(\beta^2+t^2)} (z-x+2t\xi)^2} \,.
$$
Next,
$$
\int_{\R^d} d\xi\, e^{-\xi^2/\mu} | e^{it\Delta} F_{x,\xi}(z)|^2 = \left( \frac{\beta\mu}{2(\beta^2 +t^2+{2}\beta\mu t^2)} \right)^{d/2} e^{- \frac{\beta}{{2}(\beta^2+t^2+{2}\beta\mu t^2)} (x-z)^2}
$$
and
\begin{align*}
\rho_{\gamma(t)}(z)&:=\rho_{e^{it\Delta }\gamma e^{-it\Delta }}(z)\\
 & = \iint_{\R^d\times\R^d} \frac{dx\,d\xi}{(2\pi)^d} e^{-x^2/L^2 -\xi^2/\mu} | e^{it\Delta} F_{x,\xi}(z)|^2 \\
& = \left( \frac{\beta\mu L^2}{4\pi({2\beta^2+\beta L^2 + 2t^2+4\beta\mu t^2})} \right)^{d/2} e^{-\frac{\beta}{2\beta^2+\beta L^2 + 2t^2+4\beta\mu t^2} z^2 } \,.
\end{align*}
We finally compute the $L^p_t(L^q_x)$ norm of this expression. We have
\begin{multline*}
\int_{\R^d} dz\, \rho_{\gamma(t)}(z)^q\\ = \left(\frac\pi{q}\right)^{d/2} (4\pi)^{-dq/2} \left(\mu L^2 \right)^{dq/2} \left( \frac{\beta}{{2\beta^2+\beta L^2 + 2t^2+4\beta\mu t^2}} \right)^{(q-1)d/2} 
\end{multline*}
and, using the fact that $p(q-1)d/q=2$,
\begin{align*}
\int_\R dt \left( \int_{\R^d} dz\, \rho_{\gamma(t)}(z)^q \right)^{p/q} = A_{d,p}^p \left(\mu L^2 \right)^{dp/2} \frac{\beta}{\sqrt{\beta L^2 + 2\beta^2}\ \sqrt{2+4\beta\mu} } \,,
\end{align*}
where
$$
A_{d,p}^p = \left(\frac\pi{q}\right)^{dp/(2q)} (4\pi)^{-dp/2} \int_\R \frac{ds}{1+s^2}=(q\pi)^{dp/2-1}2^{-dp}\,.
$$
Thus,
\begin{align*}
&\left( \int_\R dt \left( \int_{\R^d} dz\, \rho_{\gamma(t)}(z)^q \right)^{p/q} \right)^{1/p}\\
&\qquad\qquad = A_{d,p} \left(\mu L^2 \right)^{d/2} \left( \frac{1}{L^2 + {2}\beta} \right)^{1/(2p)} \left(\frac{1}{4\mu+2\beta^{-1}} \right)^{1/(2p)} \\
&\qquad\qquad = A_{d,p} (\mu L^2)^{d/2 - 1/(2p)} \left( 1+ \frac{{2}\beta}{L^2} \right)^{-1/(2p)} \left({ 4+ \frac{2}{\beta\mu}}\right)^{-1/(2p)} .
\end{align*}
Since $d/2-1/(2p)= d(q+1)/(4q)$, we have shown that
\begin{multline*}
\left( \int_\R dt \left( \int_{\R^d} dz\, \rho_{\gamma(t)}(z)^q \right)^{p/q} \right)^{1/p}\\ = \frac{A_{d,p}}{A_d^{(q+1)/(2q)}} N^{(q+1)/(2q)} \left( 1+ \frac{{2}\beta}{L^2} \right)^{-1/(2p)} \left( { 4+ \frac{2}{\beta\mu}} \right)^{-1/(2p)} \,. 
\end{multline*}
In a parameter regime where $1/\mu \ll \beta \ll L^2$, we obtain
$$
\left( \int_\R dt \left( \int_{\R^d} dz\, \rho_{\gamma(t)}(z)^q \right)^{p/q} \right)^{1/p} \sim \frac{{2^{-1/p}}A_{d,p}}{A_d^{(q+1)/(2q)}} N^{(q+1)/(2q)} \,,
$$
as claimed.\qed

\subsection{Proof of Proposition~\ref{prop:end-point}: the end point}
Let us define the operator
\begin{equation}
B_V:=\int_{\R}e^{-it\Delta}V(t,x)e^{it\Delta}\,dt,
\label{eq:def_A}
\end{equation}
whose kernel in Fourier space is given by 
$$\widehat{B}_V(p,q)=(2\pi)^{-\frac{d}2}\int_{\R}e^{itp^2}\cF_xV(t,p-q)e^{-itq^2}\,dt=(2\pi)^{\frac{1-d}2}\widehat{V}(p^2-q^2,p-q).$$
Here $\cF_x$ is the Fourier transform with respect to the space variable $x$ and $\widehat{V}=\cF_t\cF_xV$ denotes the Fourier transform of $V$ with respect to both space and time.
We deduce that
\begin{align*}
\tr B_V^{d+1}& =\int_{\R^d}dp\int_{\R^d}dp_1\cdots\int_{\R^d}dp_{d}\, \widehat{B}_V(p,p_1)\widehat{B}_V(p_1,p_2)\cdots\\
&\qquad\qquad\cdots \widehat{B}_V(p_{d-1},p_{d})\widehat{B}_V(p_{d}, p)\\
&=(2\pi)^{\frac{1-d^2}2}\int_{\R^d}dp\int_{\R^d}dp_1\cdots \int_{\R^d}dp_d\;\widehat{V}(p^2-p_1^2,p-p_1)\times\\
&\quad \times\widehat{V}(p_1^2-p_2^2,p_1-p_2)\cdots \widehat{V}(p_{d-1}^2-p_d^2,p_{d-1}-p_d)\widehat{V}(p_d^2-p^2,p_d-p).
\end{align*}
Note that the integral always makes sense in $[0,\ii]$ since $\widehat{V}\geq0$ by assumption. We now introduce new variables as follows:
$$\begin{cases}
k_1=p-p_1\\
k_2=p_1-p_2\\
\vdots\\
k_{d}=p_{d-1}-p_d\\
u=p+p_d.
  \end{cases}$$
A simple calculation shows that
\begin{multline*}
\tr B_V^{d+1}=2^{-\frac{(d+1)^2}{2}}\pi^{\frac{1-d^2}2}\int_{\R^d}du\int_{\R^d}dk_1\cdots \int_{\R^d}dk_d\\
\widehat{V}\big(k_1\cdot(u+k_2+\cdots +k_d),k_1\big)\widehat{V}\big(k_2\cdot(u-k_1+k_3+\cdots +k_d),k_2\big)\times\cdots\\
\cdots \times\widehat{V}\big(k_d\cdot(u-k_1-\cdots-k_{d-1}),k_d)\widehat{V}\big(-u\cdot(k_1+\cdots +k_d),-k_1-\cdots -k_d\big).
\end{multline*}
We change again variables and define $v:=K^{-1} u$ where $K$ is the matrix which contains $k_1,...,k_d$ on its rows. This matrix is such that $(K^T)^{-1}k_i=e_i$, the canonical basis. We get
\begin{multline*}
\tr B_V^{d+1}=2^{-\frac{(d+1)^2}{2}}\pi^{\frac{1-d^2}2}\int_{\R^d}dv\int_{\R^d}dk_1\cdots \int_{\R^d}dk_d\;\frac{1}{|\det K|}\\
\widehat{V}\big(v_1+k_1\cdot(k_2+\cdots +k_d),k_1\big)\widehat{V}\big(v_2+k_2\cdot(-k_1+k_3+\cdots +k_d),k_2\big)\times\cdots\\
\cdots \times\widehat{V}\big(v_d+k_d\cdot(-k_1-\cdots-k_{d-1}),k_d)\widehat{V}\big(-v_1-\cdots -v_d,-k_1-\cdots -k_d\big).
\end{multline*}
Changing again variables we get
\begin{multline*}
\tr B_V^{d+1}=2^{-\frac{(d+1)^2}{2}}\pi^{\frac{1-d^2}2}\int_{\R^d}dw\int_{\R^d}dk_1\cdots \int_{\R^d}dk_d\;\frac{1}{|\det K|}
\widehat{V}\big(w_1,k_1\big)\times\\
\times\widehat{V}\big(w_2,k_2\big)\cdots\widehat{V}\big(w_d,k_d)\widehat{V}\big(-w_1-\cdots -w_d,-k_1-\cdots -k_d\big).
\end{multline*}
Note that 
$$|\det(K)|=|k_1|\cdots|k_d|\;|\det(\omega_1,...,\omega_d)|$$ 
with $\omega_j=k_j|k_j|^{-1}$. Since $V\in L^\ii_c(\R\times\R^d)$, we have $\widehat{V}>0$ on an open set, and we see that our integral can only be finite if the function 
$$(\omega_1,...,\omega_d)\in (\mathbb{S}^{d-1})^d\mapsto |\det(\omega_1,...,\omega_d)|^{-1}$$
belongs to $L^1(\mathbb{S}^{d-1})^d$. But it is well-known that this is never the case (see, e.g.,~\cite{Drudy-88,Gressman-11,Valdimarsson-12} and the references therein).
For instance, in dimension $d=3$, we can compute in spherical coordinates
\begin{align*}
&\int_{\mathbb{S}^2}d\omega_1\int_{\mathbb{S}^2}d\omega_2\int_{\mathbb{S}^2}d\omega_3\frac{1}{|\det(\omega_1,\omega_2,\omega_3)|}\\
&\qquad=4\pi \int_{\mathbb{S}^2}d\omega_2\int_{\mathbb{S}^2}d\omega_3\frac{1}{|\det(e_3,\omega_2,\omega_3)|}\\
&\qquad=4\pi \int_0^\pi \sin\theta\, d\theta \int_0^{2\pi} d\phi\int_0^\pi \sin\theta'\, d\theta' \int_0^{2\pi} d\phi' \frac{1}{\sin\theta \sin\theta' |\sin(\phi-\phi')|}\\
&\qquad=+\ii.
\end{align*}
In the first line we have used that by rotation-invariance, the integral over $\omega_2$ and $\omega_3$ does not depend on $\omega_1$. In the second line we have written $\omega_2=(\cos\phi\,\sin\theta,\sin\phi\,\sin\theta,\cos\theta)$ and $\omega_3=(\cos\phi'\,\sin\theta',\sin\phi'\,\sin\theta',\cos\theta')$, which gives 
\begin{align*}
|\det (e_3,\omega_2,\omega_3)|&=|\cos\phi\sin\theta\sin\phi'\sin\theta'-\cos\phi'\sin\theta'\sin\phi\sin\theta|\\
&=\sin\theta\,\sin\theta'\,|\sin(\phi-\phi')|.
\end{align*}
The argument is similar in other dimensions.
\qed

\subsection{Proof of Theorem~\ref{thm:wave-op}: the wave operator}
We have already estimated $\cW_{V,\pm}^{(1)}(t_0)$ in $\gS^{2q'}$ in Theorem~\ref{thm:version_V} and the proof can be applied in the same way to bound the $\gS^{m}$ norm of $\cW_{V,\pm}^{(n)}(t_0)$, where $m=2\lceil q'/n\rceil$ is the smallest even integer which is $\geq2q'/n$. We need an even integer to have that $\norm{\cW_{V,\pm}^{(n)}(t_0)}_{\gS^m}^m=\tr \big(\cW_{V,\pm}^{(n)}(t_0)\cW_{V,\pm}^{(n)}(t_0)^*\big)^{m/2}$.
We only have to discuss the large-$n$ behavior of the constant in this estimate. To do so, we assume $n>q'$ and we look at
\begin{align*}
\norm{\cW^{(n)}_{V,+}(t_0)}^2_{\gS^2}=&\int_{t_0\leq t_1\leq\cdots \leq t_n}dt_1\cdots dt_n\int_{t_0\leq s_1\leq\cdots \leq s_n}ds_1\cdots ds_n\\ 
&\quad\tr\bigg(V\big(t_{n},x+2(t_n-t_0)p\big)\cdots \cdots V\big(t_1,x+2(t_1-t_0)p\big)\times\\
&\quad\times V\big(s_{1},x+2(s_1-t_0)p\big)\cdots \cdots V\big(s_n,x+2(s_n-t_0)p\big)\bigg).
\end{align*}
The argument is exactly the same for $\cW_{V,-}^{(n)}$.
Using~\eqref{eq:gKSS} as in the proof of Theorem~\ref{thm:version_V}, we find
\begin{multline*}
\norm{\cW^{(n)}_{V,+}(t_0)}^2_{\gS^2}
\leq (4\pi)^{-\frac{dn}{q'}}\int_{t_1\leq\cdots \leq t_n}dt_1\cdots dt_n\int_{s_1\leq\cdots \leq s_n}ds_1\cdots ds_n\\
\times\frac{v(t_1)\cdots v(t_n)\,v(s_1)\cdots v(s_n)}{|t_1-t_2|^{\frac{d}{2q'}}\cdots |t_{n}-s_n|^{\frac{d}{2q'}}|s_{n}-s_{n-1}|^{\frac{d}{2q'}}\cdots |s_1-t_1|^{\frac{d}{2q'}}}
\end{multline*}
where we have denoted $v(s):=\norm{V(s,\cdot)}_{L^{q'}(\R^d)}\1(s\geq t_0)$ for short. Now we introduce two parameters $0<\theta<1$ and $\alpha>p'$ (to be chosen later) and we write $v=v^\theta v^{1-\theta}$. By Hölder's inequality we find
\begin{multline*}
\norm{\cW^{(n)}_{V,+}(t_0)}^2_{\gS^2}\\
\leq (4\pi)^{-\frac{dn}{q'}}\left(\iint_{\substack{t_1\leq\cdots \leq t_n\\s_1\leq\cdots \leq s_n}}v(t_1)^{\theta\alpha}\cdots v(t_n)^{\theta\alpha}v(s_1)^{\theta\alpha}\cdots v(s_n)^{\theta\alpha}\right)^{\frac1{\alpha}}\\
\times\left(\iint_{\substack{t_1\leq\cdots \leq t_n\\s_1\leq\cdots \leq s_n}}\frac{v(t_1)^{(1- \theta)\alpha'}\cdots v(t_n)^{(1- \theta)\alpha'}\,v(s_1)^{(1- \theta)\alpha'}\cdots v(s_n)^{(1- \theta)\alpha'}}{|t_1-t_2|^{\frac{d\alpha'}{2q'}}\cdots |t_{n}-s_n|^{\frac{d\alpha'}{2q'}}|s_{n}-s_{n-1}|^{\frac{d\alpha'}{2q'}}\cdots |s_1-t_1|^{\frac{d\alpha'}{2q'}}}\right)^{\frac{1}{\alpha'}}
\end{multline*}
with $\alpha'=\alpha/(\alpha-1)$.
By symmetry of the integrand with respect to the times $t_j$ and $s_j$, we have
\begin{equation*}
\iint_{\substack{t_1\leq\cdots \leq t_n\\s_1\leq\cdots \leq s_n}}v(t_1)^{\theta\alpha}\cdots v(t_n)^{\theta\alpha}v(s_1)^{\theta\alpha}\cdots v(s_n)^{\theta\alpha}=\frac{1}{(n!)^2}\left(\int_\R v(t)^{\theta\alpha}\,dt\right)^{2n}.
\end{equation*}
For the other integral, we drop the time ordering for an upper bound, and we remark that it can then be written as a trace
\begin{align*}
&\iint\frac{v(t_1)^{(1- \theta)\alpha'}\cdots v(t_n)^{(1- \theta)\alpha'}\,v(s_1)^{(1- \theta)\alpha'}\cdots v(s_n)^{(1- \theta)\alpha'}}{|t_1-t_2|^{\frac{d\alpha'}{2q'}}\cdots |t_{n}-s_n|^{\frac{d\alpha'}{2q'}}|s_{n}-s_{n-1}|^{\frac{d\alpha'}{2q'}}\cdots |s_1-t_1|^{\frac{d\alpha'}{2q'}}}\\
&\qquad=(2\pi)^{-n}A^{2n}\;\tr_{L^2(\R)}\left(\frac{1}{|i\partial_t|^{\frac{1-d\alpha'/(2q')}{2}}}\,v(t)^{(1- \theta)\alpha'}\frac{1}{|i\partial_t|^{\frac{1-d\alpha'/(2q')}{2}}}\right)^{2n}\\
&\qquad=(2\pi)^{-n}A^{2n}\;\norm{\frac{1}{|i\partial_t|^{\frac{1-d\alpha'/(2q')}{2}}}\,v(t)^{\frac{(1- \theta)\alpha'}2}}_{\gS^{4n}(L^2(\R))}^{4n}.
\end{align*}
Here 
$$A=2^{\frac12(1-d\alpha'/q')}\frac{\Gamma\left(\frac{1-d\alpha'/(2q')}{2}\right)}{\Gamma\left(\frac{d\alpha'}{4q'}\right)}$$
is the constant such that $A|p|^{d\alpha'/(2q')-1}$ is the Fourier transform of $|t|^{-d\alpha'/(2q')}$.
Now we use Cwikel's inequality 
$$\norm{g(i\partial_t)f(t)}_{\gS^r_{\rm w}}\leq C_r\norm{g}_{L^r_{\rm w}}\norm{f}_{L^r},\qquad \forall r>2$$
(see~\cite{Cwikel-77} and \cite[Thm. 4.2]{Simon-79}) and get
\begin{align*}
\norm{\frac{1}{|i\partial_t|^{\frac{1-d\alpha'/(2q')}{2}}}\,v(t)^{\frac{(1- \theta)\alpha'}2}}_{\gS^{4n}}&\leq \norm{\frac{1}{|i\partial_t|^{\frac{1-d\alpha'/(2q')}{2}}}\,v(t)^{\frac{(1- \theta)\alpha'}2}}_{\gS_{\rm w}^{\frac{2}{1-d\alpha'/(2q')}}}\\
&\leq C\norm{v}_{L^{\frac{(1- \theta)\alpha'}{1-d\alpha'/(2q')}}}^{(1- \theta)\alpha'/2},
\end{align*}
provided that $4n>2(1-d\alpha'/(2q'))^{-1}$. In conclusion we have proved the inequality
\begin{equation*}
\norm{\cW^{(n)}_{V,+}(t_0)}_{\gS^2}
\leq \frac{C^n}{(n!)^{\frac1\alpha}}\norm{v}_{L^{\theta\alpha}}^{\theta n}\norm{v}_{L^{\frac{(1- \theta)\alpha'}{1-d\alpha'/(2q')}}}^{(1- \theta)n}.
\end{equation*}
In order to get our result, we have to choose $\theta$ and $\alpha$ such as to satisfy the conditions
$$0\leq \theta\leq 1,\quad 1<\alpha'<\frac{2q'}{d},\quad \theta\alpha=\frac{(1-\theta)\alpha'}{1-\frac{d\alpha'}{2q'}},\quad 4n>\frac{2}{1-\frac{d\alpha'}{2q'}}.$$
For any fixed $1<\alpha'<2q'/d$ (where we recall that $q'\geq d+2$), we can find $\theta\in(0,1)$ satisfying the above equations. A simple calculation then shows that
$\theta \alpha = p',$
as we want. Choosing $n$ larger than $2^{-1}(1-d\alpha'/(2q'))^{-1}$ finally gives
\begin{equation*}
\norm{\cW^{(n)}_{V,+}(t_0)}_{\gS^2}
\leq \frac{C^n}{(n!)^{\frac1\alpha}}\norm{v}_{L^{p'}}^{n}=\frac{C^n}{(n!)^{\frac1\alpha}}\norm{V}_{L^{p'}_t(L^{q'}_x)}^{n},
\end{equation*}
as was claimed.\qed

\bigskip

\noindent\textbf{Acknowledgement.} M.L. would like to thank Philippe Gravejat for stimulating discussions. Grants from the U.S.~NSF 
PHY-1068285 (R.F.), PHY-0965859 (E.L.), NSERC (R.S.), from the Simons Foundation \#230207 (E.L.), and from the ERC MNIQS-258023 (M.L.) are gratefully acknowledged.


\end{document}